\documentclass{amsart}
\usepackage[utf8]{inputenc}
\usepackage{amsmath,amssymb,amsthm}
\usepackage{comment}
\usepackage{hyperref}

\newcommand{\R}{\mathbb R}
\newcommand{\N}{\mathbb N}
\newcommand{\Z}{\mathbb Z}

\newcommand{\eps}{\varepsilon}

\newcommand{\loc}{\mathrm{loc}}

\renewcommand{\vec}[1]{\mathbf{#1}}

\newcommand{\enclose}[1]{\left(#1\right)}

\newcommand{\ENCLOSE}[1]{\left\{#1\right\}}

\usepackage{xcolor}

\newtheorem{theorem}{Theorem}[section]
\newtheorem{proposition}[theorem]{Proposition}

\newtheorem{lemma}[theorem]{Lemma}
\newtheorem{corollary}[theorem]{Corollary}

\theoremstyle{definition}
\newtheorem{definition}[theorem]{Definition}
\theoremstyle{remark}
\newtheorem{remark}[theorem]{Remark}

\author[Novaga]{Matteo Novaga}
\address[Matteo Novaga, Emanuele Paolini, Vincenzo Tortorelli]{Dipartimento di Matematica, Universit\`a di Pisa \\
	Largo Bruno Pontecorvo 5 \\ I-56127, Pisa}
\email[Matteo Novaga]{matteo.novaga@unipi.it}

\author[Paolini]{Emanuele Paolini} 
\email[Emanuele Paolini]{emanuele.paolini@unipi.it}

\author[Stepanov]{Eugene Stepanov}
\address[Eugene Stepanov]{St.Petersburg Branch of the Steklov Mathematical Institute of the Russian Academy of Sciences,
	Fontanka 27,
	191023 St.Petersburg, Russia
	\and
	Faculty of Mathematics, Higher School of Economics, Moscow
}
\email[Eugene Stepanov]{stepanov.eugene@gmail.com}

\author[Tortorelli]{Vincenzo Maria Tortorelli}
\email[Vincenzo Maria Tortorelli]{tortorelli@dm.unipi.it}

\thanks{The first author is a member of the INDAM/GNAMPA and was supported by the PRIN Project 2019/24.
	The work of the third author has been partially financed by
	the RFBR grant \#20-01-00630 A.
}
\date{\today}

\title[Isoperimetric clusters in homogeneous spaces]{Isoperimetric clusters in homogeneous spaces via concentration compactness}

\begin{document}
	\begin{abstract}
		We show the existence of generalized clusters of a finite or even infinite number of sets, with minimal total perimeter and given total masses, in metric measure spaces homogeneous with respect to a group acting by measure preserving homeomorphisms, 
		for a quite wide range of perimeter functionals.
		Such generalized clusters are a natural ``relaxed'' version of a cluster and can be thought of as ``albums'' with possibly infinite pages, having a minimal cluster drawn on each page, the total perimeter and the vector of masses being calculated by summation over all pages, the total perimeter being minimal among all generalized clusters with the same masses. The examples include any anisotropic perimeter in a Euclidean space, as well as a hyperbolic plane with the Riemannian perimeter and Heisenberg groups with a canonical left invariant perimeter or its equivalent versions.
	\end{abstract}
	
	\maketitle
	
	\section{Introduction}
	
	
	In a finite-dimensional Euclidean space $\R^n$ we call 
	$\vec E = (E_1,\dots, E_N)$ with $N\in \N\cup\{\infty\}$ an \textit{$N$-cluster} of sets 
	if each $E_j\subset \R^n$ is Borel (possibly empty), and
	$\mathcal{L}^n(E_i\cap E_j)=0$ for all $i\neq j$,  $\mathcal{L}^n$ standing for the Lebesgue measure.
	If $P(M)$ stands for the classical (Euclidean) perimeter of the set $M\subset \R^n$, we set
	\begin{align*}
	\vec m(\vec E) &:= \enclose{\mathcal{L}^n(E_1),\dots,\mathcal{L}^n(E_N)} \\
	P(\vec E) &:= 
	\frac 1 2 \sum_{j\geq 1} P(E_j)+
	\frac 1 2 P\left(\bigcup_{j\geq 1} E_j\right) 
	\end{align*}
	An $N$-cluster $\vec E$ is usually called
	%
	\emph{minimal},  or \textit{isoperimetric}, if 
	\[
	P(\vec E) = \inf\ENCLOSE{P(\vec F)\colon \vec m(\vec F)=\vec m(\vec E)}.  
	\]
	
	Existence of \textit{finite} (i.e. with $N <+\infty$) isoperimetric clusters for the classical (Euclidean) perimeter in a Euclidean space 
	has been proven in~\cite{Maggi12-GMTbook} for each given vector of masses $\vec m$. The technique of the proof has been 
	further extended to several other situations, for various different notions of perimeter. More has been done for the problem of existence of \textit{isoperimetric sets}, i.e.\ in our terminology, minimal $1$-clusters. It has been accomplished
	for different versions of perimeter in finite-dimensional spaces such as Riemannian manifolds 
	(see the list of relevant results in the Introduction to~\cite{MondinoNard16-isoperim} as well as the recent papers~\cite{CeNo18,FloreNard20-isoperimRiemann,AntFogPoz21}) and even in more general spaces with Ricci curvature bounds~\cite{AnBrFoPo21,AntPasPoz21}. However these techniques  are based on regularity of isoperimetric
	sets, and are quite difficult even in the Euclidean case. Moreover, such results are only valid for \textit{finite} clusters. 
	
	The difficulty to prove the existence by purely variational techniques, even
	in the Euclidean space and just for isoperimetric sets (i.e. just $1$-clusters) is only due to noncompactness of the ambient space
	(in compact spaces it is trivial), and comes from the fact
	that minimizing sequences of sets may ``escape to infinity'' and lose part of their volumes. It is however well-known that
	the limits of the pieces going to infinity are still isoperimetric sets for their own volumes (which may be lower than the original requirement), see e.g.~\cite{DePhilFranzPrat17,PratSar18-isoperim} where the problem of existence of isoperimetric sets has been studied for
	quite a general family of weighted perimeters and volume measures. The same is known to happen also with \textit{finite}  isoperimetric clusters, see~\cite{Scattaglia20-clust} where a formula for the  perimeter of a minimal finite cluster is proven (again, for possibly weighted perimeters and volumes) suggesting that a minimal cluster is given by the union of the limit cluster (with possibly lower volumes) for a minimizing sequence
	and a ``cluster at infinity'' which has precisely the missing volumes. 
	Therefore to prove the existence of an isoperimetric set or cluster one needs to appropriately adjust the sets by adding necessary volume which requires a great deal of heavy regularity techniques. However it is important to emphasize that the problem is not merely technical. In fact, even for $N=1$ (isoperimetric sets) it might happen that in some Riemannian manifolds there are no  isoperimetric sets of some or even every volume, see~\cite{Ritore01-isoperim-symm} as well as~\cite{Nardulli14_isoperimgener,MondinoNard16-isoperim}.
	
	Here we consider a more general situation of a metric measure space $(X,d,\mu)$  (with distance $d$ and measure $\mu$) instead of an Euclidean space with Lebesgue measure, which is \textit{homogeneous} with respect to some group acting by measure preserving homeomorphisms. 
	We show that in such a situation one can easily prove by means of concentration compactness-like technique (i.e. similar to~\cite{Lions84-conccomp}), the existence of ``generalized'' minimal clusters (finite or possibly infinite) for each given vector of masses. Such generalized clusters are in a sense a natural ``relaxed'' version of a cluster and can be thought of as ``albums'' with possibly infinite pages, having a minimal $N$-cluster drawn on each page, the total perimeter and the vector of masses being calculated by summation over all pages, the total perimeter being minimal among all generalized clusters with the same vector of masses. Such a generalized cluster in fact keeps track of all the parts of minimal sequences ``escaping at infinity'' by placing them on possibly different pages.
	
	The toy examples we provide include finite or infinite minimal clusters for any anisotropic perimeter in $\R^n$, as well as in a hyperbolic plane with the Riemannian perimeter and in Heisenberg groups with a canonical left invariant perimeter or its equivalent versions.  In all these examples there is some natural discrete group acting on the respective space properly discontinuously and cocompactly by measure preserving isometries (of course, the existence of isoperimetric sets is also known in such situations).
	Note that the technique developed here is purely variational and does not involve any regularity-type arguments, thus allowing to obtain existence of generalzied minimal clusters for the whole range of perimeter functionals. The cost paid for such a simplicity is of course that one can say nothing a priori about the regularity of generalized clusters; and it has to be noted that even very weak regularity properties of a minimal generalized cluster would allow to conclude 
	the existence of a solution to the original problem (i.e.\ the existence of minimal clusters in the original space).
	
	The paper is structured as follows. After introducing the general axiomatic notion of a perimieter functional by a set of quite weak requirements
	in Section~\ref{sec:notation}, we provide in Section~\ref{sec:compsets} the general semicontinuity and compactness result for this functional
	for sequences of sets and in Section~\ref{sec:compclusters} for sequences of clusters. This allows to prove the existence of generalized
	optimal clusters in Section~\ref{sec:existgencluster}. Finally, in Section~\ref{sec_ClusterExamples} we give some basic applications of such existence results for various types of perimeter functional in different environments, in partcicular, in a finite-dimensional normed space (i.e.\ with either classical or anisotropic perimeter), in a hyperbolic space and in te Heisenberg group.
	
	\section{Notation and prelimiaries}
	\label{sec:notation}
	
	In the sequel we suppose $(X,d,\mu)$ to be a  metric measure space with distance $d$ and nonnegative $\sigma$-finite Radon measure
	$\mu$ with $\mu(X)\neq 0$, and $G$ be a discrete, countable, topological group acting properly discontinuously on $X$ 
	(i.e.\ $\{g\in G\colon gK\cap K \neq \emptyset\}$ is finite for all compact $K\subset X$) by homeomorphisms preserving the measure $\mu$. 
	We denote by $\mathcal{B}(X)$ the Borel $\sigma$-algebra of $X$, and by $\mathcal{A}(X)$ the class of open subsets of $X$.
	The metric space $(X,d)$ is said to have Heine-Borel property, if every closed ball is compact.
	
	By $L^1(\mu)$ (resp.\ $L^1_\loc(\mu)$) we denote the usual Lebesgue space of $\mu$-integrable (resp.\ $\mu$-integrable over compact sets) functions over $X$, and if $X=\R^n$, we denote by $C_0^\infty(\R^n)$ the set of infinitely differentiable functions with compact support. As usual, for an $E\subset X$ we let $\bar E$ stand for the closure of $E$ and $\mathbf{1}_E$ for the characteristic function of $E$, while $B_r(x)$ denotes the open ball of radius $r>0$ centered at $x\in X$.
	
	We write $g_k\to\infty$ for a sequence of $g_k\in G$, if $\lim_k g_k=\infty$ in the one-point compactifictaion
	$G\cup\{\infty\}$ of $G$, that is, for every finite $F\subset G$ the set of $\{k\in \N\colon g_k\in F\}$ is finite. 
	
	In what follows we let $P\colon \mathcal{B}(X)\times \mathcal{A}(X)\to \R^+\cup\{\infty\}$ be a ``relative perimeter'' functional 
	for which we customarily abbreviate $P(E):=P(E,X)$,
	satisfying
	\begin{itemize}
		\item[(semicontinuity)] $P(D, U) \le \liminf_k P(D_k,U)$ whenever $D_k \stackrel{L^1_\loc(\mu)}\longrightarrow D$,
		\item[(Beppo Levi)] If $U_k \nearrow X$, then $P(D, U_k)\nearrow P(D)$ as $k\to\infty$,
		\item[(monotonicity)] $P(B, U)\leq  P(B, V)$ as $U\subset V$,
		\item[(superadditivity)] $P(B,U)\geq \sum_{k=1}^M P(B, U_k)$ whenever $U=\bigsqcup_{i=1}^M  U_k$,
		\item[($G$-invariance)] $P(gB,gU)= P(B,U)$ for every $g\in G$,
		\item[(compactness)] 
		if $E_k\subset X$ satisfy $\displaystyle\sup_k P(E_k,U)<+\infty$ 
		for some precompact set $U$, then
		there exists an $E\subset X$ such that, up to a subsequence, 
		$E_k \cap U\stackrel{L^1(\mu)}\longrightarrow E$.
	\end{itemize}

	\section{Compactness and semicontinuity for sequences of sets}
	\label{sec:compsets}
	
	\subsection{Semicontinuity}
	
	We start with the following semicontinuity statement which will be further applied both to perimeters and measures.
	
	\begin{theorem}[semicontinuity]
		\label{th:comp_semicont}
		Assume that $(X,d)$ has Heine-Borel property.
		Let $E_k\subset X$ be a sequence of Borel sets, and 
		$E^i\subset X$, $i\in I$, $I$ at most countable, $g_k^i\in G$ such that:
		\begin{gather}
		\label{eq_diffg1}
		\lim_k (g_k^{i'})^{-1}g_k^i = +\infty \qquad \mbox{for all $i\neq i'$}\\
		\label{eq_convE1}	
		(g_k^{i})^{-1} E_k \stackrel{L^1_\loc(\mu)}\longrightarrow E^i \qquad \text{as $k\to+\infty$}.
		\end{gather}
		Let also $F\colon \mathcal{B}(X)\times \mathcal{A}(X)\to \R$ be a functional satisfying
		\begin{itemize}
			\item[(semicontinuity)] $F(B, U) \le \liminf_k F(B_k,U)$ whenever $B_k \stackrel{L^1_\loc(\mu)}\longrightarrow B$,
			\item[(Beppo Levi)] If $U_k \nearrow X$, then $F(D, U_k)\nearrow F(D, X)$ as $k\to\infty$,
			\item[(monotonicity)] $F(D, U)\leq  F(D, V)$ as $U\subset V$,
			\item[(superadditivity)] $F(B,U)\geq \sum_{k=1}^M F(B, U_k)$ whenever $U=\bigsqcup_{i=1}^M  U_k$,
			\item[($G$-invariance)] $F(gB,gU)= F(B,U)$ for every $g\in G$.
		\end{itemize}
		If $G$ acts on $X$ properly discontinuously, then
		\begin{gather*}
		\sum_i F(E^i,X) \le \liminf_k F(E_k,X).
		\end{gather*}
	\end{theorem}
	
	\begin{proof}
		Up to a subsequence (not relabeled) we may
		suppose that \[\liminf_k F(E_k,X)=\lim_k F(E_k,X).\]
		Fix arbitrary $M\in \N$, $x_0\in X$ and  $R>0$ and denote for brevity $U:=B_R(x_0)$.  
		Since $U$ is precompact by Heine-Borel property, and the action of $G$ on $X$ is properly discontinuous, then  
		$gU\cap U\neq \emptyset$ only for a finite number of $g\in G$. Thus
		from~\eqref{eq_diffg1} we get  that the sets 
		$g_k^1 U, \dots, g_k^M U$ are pairwise disjoint for all sufficiently large $k$.
		Hence we get
		\begin{equation}\label{eq_semicontF1}
		\begin{aligned}
		\lim_k F(E_k,X)  & \geq \limsup_k F\left( E_k,\bigsqcup_{i=1}^M g_k^i U \right) \quad \mbox{by monotonicity of $F$}\\
		& = \limsup_k \sum_{i=1}^M F(E_k, g_k^i U) \quad \mbox{by superadditivity of $F$}\\
		& \geq \sum_{i=1}^M \liminf_k F(E_k, g_k^i U).
		\end{aligned}
		\end{equation}	But $F(E_k,g_k^i U)) = F((g_k^i)^{-1} E_k,U)$ and from~\eqref{eq_convE1}  
		using semicontinuity of $F$ we obtain
		\begin{equation}\label{eq_semicontF2}
		\liminf_k F(E_k,g_k^i U) = \liminf_k F((g_k^i)^{-1} E_k,U)\ge F(E^i,U).
		\end{equation}	
		The inequalities~\eqref{eq_semicontF1} and\eqref{eq_semicontF2} together give
		\[
		\lim_k F(E_k, X)\ge \sum_{i=1}^M F(E^i,U) 
		=\sum_{i=1}^M F(E^i,B_R(x_0)).
		\]
		Letting now $R\to+\infty$, we obtain 
		\[
		\lim_k F(E_k,X)\ge \sum_{i=1}^M F(E^i,X)  
		\]
		and finally letting $M\to +\infty$ we get 
		\[
		\lim_k F(E_k,X) \ge \sum_i F(E^i, X). 
		\]
		as desired.
	\end{proof}
	
	\begin{corollary}\label{co_comp_semicont1}
		Let $(X,d,\mu)$, $E_k$, 
		$E^i$, $I$, $g_k^i$ be as in Theorem~\ref{th:comp_semicont}.
		If $G$ acts on $X$ properly discontinuously, then
		\begin{align*}
		\sum_i P(E^i) &\le \liminf_k P(E_k),\\
		\sum_i \mu(E^i) &\le \liminf_k \mu(E_k).
		\end{align*} 
	\end{corollary}
	
	\begin{proof}
		Apply Theorem~\ref{th:comp_semicont} with $F(B,U):=P(E,U)$, and then with $F(B,U):=\mu(E\cap U)$.
	\end{proof}
	
	\subsection{Compactness}
	
	The following theorem is our main technical tool to prove the existence of generalized isoperimetric clusters.
	
	\begin{theorem}[concentration compactness]
		\label{th:comp_comp}
		Assume that $(X,d)$ has Heine-Borel property.
		Let 
		\[E_k\in \mathcal{B}(X), \quad \mu(E_k)=m,\quad \sup_k P(E_k) = P < +\infty.\]
		Suppose that 
		\begin{itemize}
			\item[(i)] $G$ act on $X$ properly discontinuously,
			\item[(ii)] there is a precompact $B\in \mathcal{B}(X)$ with $\mu(\partial B)=0$, 
			such that 
			$GB=X$ and $\mu(g B \cap B)=0$ for all $g\in G$ except $g=1$.
		\end{itemize}
		Assume further that 
		\begin{itemize}
			\item[(iii)] there is an $\varepsilon>0$ and a precompact open $V\subset X$ with 
			$B\subset V$ such that 
			the local isoperimetric inequality 
			\begin{equation}
			\label{eq_relisoperim1a}
			P(E,V)\geq f(\mu(E\cap V))
			\end{equation}	
			holds for all Borel $E\subset X$ 
			with $\mu(E\cap V)\leq \varepsilon$ with some nondecreasing function $f\colon\R^+\to \R^+$ such that
			$f(0)=0$ and $f'(0)=+\infty$.
		\end{itemize}		
		Then there are: a subsequence $E_k$, some Borel sets
		$E^i\subset X$, 
		and $g_k^i\in G$, $i\in I$, $I$ at most countable, 
		such that
		\begin{gather}
		\label{eq_diffg2}
		\lim_k (g_k^{i'})^{-1}g_k^i = +\infty \qquad \mbox{for all $i\neq i'$}\\
		\label{eq_convE2}	
		(g_k^i)^{-1} E_k \stackrel{L^1_\loc(\mu)}\longrightarrow E^i,
		\qquad k\to +\infty\\
		\label{eq_convE3}	
		\sum_i \mu(E^i) = m.
		\end{gather}
	\end{theorem}
	
	\begin{remark}\label{rm_isoperim1}
		In many examples condition~(iii) of the above Theorem~\ref{th:comp_comp} is verified with $f(t):=Ct^\alpha$ for some $\alpha \in (0,1)$ and $C>0$, namely,~\eqref{eq_relisoperim1a} reads as 
		the local isoperimetric inequality 
		\begin{equation}
		\label{eq_relisoperim1}
		P(E,V)\geq C\mu(E\cap V)^\alpha.
		\end{equation}	
	\end{remark}
	
	\begin{proof}
		For each $k$ let $h_k^j$ be an enumeration of $G$ 
		such that  
		the sequence  
		$j\mapsto \mu(E_k \cap h_k^j V)$ is not-increasing. Let $U$ stand for the interior of $B$. 
		One observes that $U\neq \emptyset$, because $\mu(B)>0$ (otherwise from $GB=X$ one would have $\mu(X)=0$), and $\mu(\partial B)=0$ hence $\mu(U)= \mu(B\setminus \partial B)>0$.
		
		Note that 
		$Q_k^j:= h_k^j U$ are open sets since $G$ acts by homeomorphisms.
		If there are infinitely many sets $h_k^j V$ in which $E_k$ 
		has positive measure we avoid to enumerate those 
		with zero measure.
		
		By the compactness property of the perimeter functional,
		for each $j$ we may suppose that up to a subsequence
		$((h_k^j)^{-1}E_k) \cap U$ converges in $L^1(\mu)$ as $k\to +\infty$ 
		and hence there exists $F^j\subset U$ such that 
		\begin{equation}\label{eq_gkjEk1}
		(h_k^j)^{-1} (E_k \cap Q_k^j) = ((h_k^j)^{-1}E_k)\cap  U \stackrel{L^1(\mu)}\longrightarrow F^j,
		\qquad \text{as $k\to \infty$}.
		\end{equation}
		By a diagonal argument we may choose a single subsequence of $k$ (not relabeled) such that~\eqref{eq_gkjEk1} holds for all $j\in \N$.
		The rest of the proof is divided in several steps.
		
		{\sc Step 1}. We claim that $\sum_j \mu(F^j) = m$.
		
		To show the claim, we first note that
		since the action of $G$ on $X$ is properly discontinuous, we have 
		\[
		\nu:=\#\{g\in G\colon g\bar V\cap \bar V\neq \emptyset\} <\infty.
		\]
		Applying Lemma~\ref{lm_groupPart1} with $K:=\bar V$
		one has a partition of $G$ is at most $\nu$ disjoint subsets $\mathcal{F}_j$ such that the sets 
		$h\bar V$ are mutually disjoint for all $h\in \mathcal{F}_j$.
		Thus for each fixed $k\in \N$ 
		one can find a finite partition $\N = \bigsqcup_{\lambda=1}^\nu J_{k,\lambda}$ of all indices $j$ 
		such that for $j \in J_{k,\lambda}$ the sets
		$h_k^j \bar V$ are pairwise disjoint, by setting
		$J_{k,\lambda} := \{j\in \N\colon h_k^j \in \mathcal{F}_\lambda\}$.
		
		Since $j\mapsto \mu(E_k\cap h_k^j(V))$ is non-increasing, we have
		\begin{equation}\label{eq_bdmu1}
		\begin{aligned}
		\mu(E_k \cap h_k^n V)
		&\leq \frac 1 n \sum_{j=1}^n \mu(E_k \cap h_k^j V) 
		\leq \frac 1 n \sum_{j=1}^{+\infty} \mu(E_k \cap h_k^j V) \\
		&=
		\frac 1 n \sum_{\lambda =1}^\nu \sum_{j\in J_{k,\lambda}} \mu(E_k \cap h_k^j V)
		\leq \frac \nu n \mu(E_k) = \frac{\nu m}{n}.
		\end{aligned}
		\end{equation}
		In particular, by~\eqref{eq_bdmu1} for every $\delta>0$
		one can choose an
		$N=N(\varepsilon,\delta )\in \N$ such that for all $j\ge N$ we have 
		\begin{eqnarray}
		\label{eq_mukeps1}
		\mu(E_k \cap h_k^j V)  \leq \varepsilon,&\quad\text{hence also}\\
		\label{eq_mukeps2}
		\mu(E_k \cap Q_k^j)  \leq \varepsilon,
		\end{eqnarray}
		and, moreover,
		\begin{equation}\label{eq_mukeps3}
		\frac{\mu(E_k \cap Q_k^j)}{f(\mu(E_k \cap Q_k^j))}\leq \delta.
		\end{equation}
		Thus for such $j\in \N$, using~\eqref{eq_mukeps2},~(iii) and the invariance of perimeter,
		we get
		\begin{equation}\label{eq_bdper1}
		\begin{aligned}
		f(\mu(E_k \cap Q_k^j))
		&\le f(\mu((h_k^j)^{-1}E_k \cap V))\\
		&\leq P((h_k^j)^{-1}E_k ,V) 
		=
		P(E_k ,h_k^j V) 
		.
		\end{aligned}
		\end{equation}
		Therefore, for $n\geq N$ we have
		\begin{equation*}\label{eq_summu1}
		\begin{aligned}
		\sum_{j=n}^{+\infty} \mu(E_k \cap Q_k^j)
		&\le \sum_{j=n}^{+\infty} \frac{\mu(E_k \cap Q_k^j)}{f(\mu(E_k \cap Q_k^j))} 
		f(\mu(E_k \cap Q_k^j)) \\
		&\le \sum_{j=n}^{+\infty} \frac{\mu(E_k \cap Q_k^j)}{f(\mu(E_k \cap Q_k^j))}
		P(E_k,h_k^jV) \quad\mbox{[by~\eqref{eq_bdper1}]}\\
		&\le \delta
		\sum_{j=n}^{+\infty} P(E_k,h_k^j V) \quad\mbox{[by~\eqref{eq_mukeps3}]}\\
		&\le \delta
		\sum_{\lambda=1}^\nu \sum_{j\in J_{k,\lambda}} P(E_k,h_k^j V) \\
		& \leq \delta
		\sum_{\lambda=1}^\nu  P(E_k,\bigsqcup_{j\in J_{k,\lambda}} h_k^j V) \quad\mbox{[by superadditivity of perimeter]}\\ 
		& \leq \delta
		\sum_{\lambda=1}^\nu P(E_k)\quad\mbox{[by monotonicity of perimeter]}\\
		& =  \nu \delta
		P(E_k) \leq  \nu\delta P.
		\end{aligned}
		\end{equation*}
		Since $\delta>0$ is arbitrary,
		we get
		\[	
		\lim_{n\to +\infty}	
		\sup_k \sum_{j=n}^{+\infty} \mu(E_k \cap Q_k^j) = 0.
		\]
		But we also have 
		\[\mu(F^j) = \lim_k \mu(E_k \cap Q_k^j).\]
		Apply Lemma~\ref{lm:equisummability} with 
		$m_j:=\mu(F^j)$, $m_{k,j}:=\mu(E_k \cap Q_k^j)$,
		recalling that
		\begin{align*}
		\sum_i m_{k,j} 
		&= \mu\big( E_k \cap \bigcup_j h_k^j U \big) 
		\qquad \text{[since $\mu(hU\cap U)=0$ for all $h\in G$]}
		\\
		&= \mu\big( E_k \cap \bigcup_j h_k^j B \big) 
		\quad\text{[since $\mu(h(B\setminus U))=\mu(B\setminus U)=0$ for all $h\in G$]}\\ 
		&=\mu(E_k)=m\qquad \text{[since $GB=X$]},
		\end{align*}
		to obtain $\sum_j \mu(F^j) = m$ as claimed.

		{\sc Step 2: construction of $g_k^i$ and $E^i$}. 
		Define on $\N$ an equivalence relation by letting 
		$j\sim j'$ whenever the set $\{(h_k^{j'})^{-1}h_k^j\colon k\in \N\}\subset G$ is finite
		and let $I:=\N/{\sim}$ be the quotient set.
		For each $i\in I$ let $\underline{i}:= \min i$.
		Passing to a subsequence in $k$ we might and shall suppose that for all $i\in I$ and all $j\in i$ 
		the sequence $\{(h_k^{\underline i})^{-1} h_k^j\}_k$ (which by assumption assumes a finite number of values) is actually constant, which will be denoted $h^j\in G$.
		
		Clearly for $j\in i$ all $h^j$ are distinct
		because so are $h_k^j\in G$.
		Define $g_k^i := h_k^{\underline{i}}$ 
		and suppose that, up to a subsequence,
		the sets $(g_k^i)^{-1}E_k$ converge in $L^1_\loc(\mu)$ 
		to some Borel set $E^i$ as $k\to +\infty$.
		Hence~\eqref{eq_diffg2} and~\eqref{eq_convE2} hold by construction.
		
		{\sc Step 3}.  It remains to prove~\eqref{eq_convE3}. To this aim observe that from~\eqref{eq_gkjEk1} with  $j\in i$ we get
		\begin{align*}
		\mu(F^j) &\leftarrow
		\mu(E_k \cap Q_k^j)=  \mu\left((g_k^i)^{-1} (E_k \cap Q_k^j)   \right) \\
		&= \mu\left((g_k^i)^{-1} E_k \cap (g_k^i)^{-1} h_k^j U\right)\\
		&= \mu\left((g_k^i)^{-1} E_k \cap h^j U\right)\\
		&\to  \mu\left(E^i\cap h^j U\right) 
		\end{align*}
		as $k\to\infty$, so that $\mu(F^j) =\mu(E^i\cap h^j U)$ whenever $j\in i$.
		Therefore
		\begin{align*}
		\lim_k \mu(E_k) &= m = \sum_j \mu(F^j) \quad\text{[by Step~1]}\\
		& = \sum_i \sum_{j\in i} \mu(E^i\cap h^j U) = \sum_i \mu\big(E^i\cap \bigcup_{j\in i}h^jU\big) \leq \sum_i \mu\left(E^i\right),
		\end{align*}
		the reverse inequality coming from Theorem~\ref{th:comp_semicont}.
	\end{proof}

	\section{Compactness and semicontinuity for sequences of clusters}
	\label{sec:compclusters}
	
	The following definition of an isoperimetric cluster is an obvious extension 
	of the classical one provided in the Introduction from a Euclidean to a general metric measure space.
	
	\begin{definition}[isoperimetric clusters]
		We say that $\vec E = (E_j)_{j=1}^N$ with $N\in \N\cup\{\infty\}$, is an $N$-cluster in $X$ 
		if each $E_j\subset X$ is a Borel set and
		$\mu(E_i\cap E_j)=0$ for all $i\neq j$. 
		If $N$ is finite we have $\vec E = (E_1,\dots, E_N)$, if $N=\infty$ 
		we have $\vec E = (E_1,\dots,E_n,\dots)$.
		We set
		\begin{align*}
		\vec m(\vec E) &:= \enclose{\mu(E_j)}_{j=1}^N \\
		P(\vec E) &:= 
		\frac 1 2 \sum_{j\geq 1} P(E_j)+
		\frac 1 2 P\left(\bigcup_{j\geq 1} E_j\right)
		\end{align*}
		An $N$-cluster $\vec E$
		will be called
		%
		\emph{minimal},  or \textit{isoperimetric}, if 
		\[
		P(\vec E) = \inf\ENCLOSE{P(\vec F)\colon \vec m(\vec F)=\vec m(\vec E)}.  
		\]
	\end{definition}
	
	\begin{theorem}[concentration compactness for clusters]\label{th_comp_clust}
		Assume that $(X,d)$ has Heine-Borel property.
		Let $\vec E_k$ be a sequence of $N$-clusters in $X$, 
		$N\in \N \cup \{\infty\}$,
		with 
		\[\vec m(\vec E_k) = \vec m\in \R^N\quad\mbox{and }
		P(\vec E_k)\le P <+\infty.
		\]	
		Then under conditions~(i),~(ii) and~(iii) of Theorem~\ref{th:comp_comp} 
		there exist  $N$-clusters $\vec E^i$ 
		with $i\in I$, $I$ at most countable, such that 
		\begin{gather*}
		\sum_i \vec m(\vec E^i) = \vec m\\
		\sum_i P(\vec E^i) \le \liminf_k P(\vec E_k)
		\end{gather*}
		and there exist $g_k^i\in G$ such that up to a subsequence
		\[
		(g_k^i)^{-1}(E_k)_j  \to (E^i)_j\qquad \text{as $k\to +\infty$ in $L^1_\loc(\mu)$}  
		\]
		for all $j=1,\ldots, N$ and $i\in I$.
	\end{theorem}
	
	\begin{proof}
		Up to a subsequence suppose that $\liminf_k P(\vec E_k) = \lim_k P(\vec E_k)$.
		Let 
		\[
		F_k := \bigsqcup_{j=1}^N (E_k)_j.
		\]
		Clearly $P(F_k)\le 2 P(\vec E_k)\leq 2P$, and hence we can apply Theorem~\ref{th:comp_comp}
		to find $F^i\subset X$, $g_k^i\in G$, $\lim_k (g_k^{i^\prime})^{-1}g_k^i= \infty$ for $i\neq i^\prime$  
		such that, up to a subsequence,
		$(g_k^i)^{-1}F_k\to F^i$ in $L^1_\loc(\mu)$ as $k\to\infty$ with 
		$\displaystyle \sum_i \mu(F^i)=\mu(F_k) = \sum_{j=1}^N m_j$. 
		Since $P((E_k)_j) \le P(\vec E_k) \le P$,
		by the compactness assumption on the perimeter functional, up to  subsequences we can 
		define the sets 
		\[E^i_j := L^1_\loc(\mu)\text{-}\lim_k (g_k^i)^{-1}(E_k)_j.\]
		By Lemma~\ref{lm:union} then 
		\[	
		F^i = \bigsqcup_{j=1}^N E^i_j.
		\]	
		By Theorem~\ref{th:comp_semicont} we have $\displaystyle \sum_i \mu(E^i_j)\le \mu((E_k)_j)$. 
		Therefore, since
		\[	
		\sum_j\sum_i \mu(E^i_j)=\sum_i\sum_j \mu(E^i_j)=\sum_i\mu(F^i)=\mu(F_k)=\sum_j \mu((E_k)_j),
		\] 
		one gets $\displaystyle \sum_i \mu(E^i_j)=\mu((E_k)_j)$.
		Hence $\vec E^i = (E^i_1,\dots,E^i_N)$ is an $M$-cluster, $M\le N$,  with	
		\begin{align*}
		\sum_i  \vec m(\vec E^i) &= \vec m,\\
		P(\vec E^i) &=\frac 1 2 \sum_{j=1}^N P(E^i_j) + \frac 1 2 P(E^i).
		\end{align*}
		Again by Theorem~\ref{th:comp_semicont}
		we have 
		\begin{align*}
		\sum_i P(F^i) &\le \liminf_k P(F_k), \\
		\sum_i P(E^i_j) &\le \liminf_k\sum_j P((E_k)_j),\qquad j=1,\dots,N
		\end{align*}
		Summing up these relationships we get
		\begin{align*}
		\sum_i P(\vec E^i) &\le \lim_k P(\vec E_k)
		\end{align*}	
		as claimed.
	\end{proof}
	
	\section{Existence of generalized isoperimeric clusters}
	\label{sec:existgencluster}
	
	We introduce now the notion of a \textit{generalized isoperimeric cluster}.
	
	\begin{definition}[generalized isoperimetric clusters]
		Let 
		$X^\infty:= \Z\times X$ 
		and let $\phi_i\colon X\to X^\infty$ be the inclusion $\phi_i(x) = (i,x)$.
		We call 
		$\vec E = (E_i)_{i=1}^N\subset X^\infty$ a generalized $N$-cluster in $X$
		with $N\in \N\cup\{\infty\}$, if for each $j\in \Z$ the vector $\vec E^j$ with components 
		$E^j_i:=\phi_j^{-1}(E_i) \subset X$, $i=1,\ldots, N$
		is a $N$-cluster in $X$. 
		We set
		\begin{align*}
		\vec m^\infty(\vec E) &:= \sum_j \vec m(\vec E^j) = \enclose{\sum_j\mu(E_1^j),\dots,\sum_j \mu(E_N^j)} \\
		P^\infty(\vec E) &:= \sum_j P(\vec E^j). 
		\end{align*}
		A generalized $N$-cluster $\vec E\in X^\infty$ will be called
		\emph{minimal},   or \textit{isoperimetric}, if 
		\[
		P^\infty(\vec E) = \inf\ENCLOSE{P^\infty(\vec F)\colon \vec m(\vec F)=\vec m(\vec E)}.  
		\]
	\end{definition}
	
	It is convenient to think of a generalized cluster as a sequence of $N$-clusters $\vec E^i$ 
	drawn each on a page $X_i:=\{i\}\times X$ of the ``album'' $X^\infty$
	and $\phi_i^{-1}$ as an extraction of the $i$-th page from this album.
	
	The following immediate observation is worth mentioning.
	
	\begin{proposition}\label{prop_genclustMinpage1}
		Every page of a generalzied minimal cluster is a minimal cluster, namely, 
		for every $j\in \N$ one has that $\phi_j^{-1}(\vec E)$ is a minimal cluster. 
	\end{proposition}
	
	\begin{proof}
		If not, there is an $N$-cluster $\vec F^j$ such that 
		\[
		P(\vec F^j) <  P(\phi_j^{-1}(\vec E))
		\quad\mbox{and } \vec m(\vec F^j)=\vec m(\phi_j^{-1}(\vec E)).
		\]
		Then for the generalized cluster $\vec F$ defined by
		\begin{align*}
		\phi_k^{-1}(\vec F) & = \phi_k^{-1}(\vec E),\qquad k\neq j,\\
		\phi_j^{-1}(\vec F) & = \vec F^j, 
		\end{align*}
		one has $P^\infty(\vec F)<P^\infty(\vec E)$ contradicting the minimality of $\vec E$.
	\end{proof}
	
	We show now the existence of minimal generalized clusters. 
	
	\begin{theorem}[existence of generalized minimal clusters]\label{th_exist_generclust1}
		Assume that $(X,d)$ has Heine-Borel property.
		If conditions~(i),~(ii) and~(iii) of Theorem~\ref{th:comp_comp} hold,
		condition~(iii) being satisfied with $f\colon \R^+\to \R^+$ subadditive
		(this is true in particular when $f(t):= Ct^\alpha$, with $C>0$, $\alpha\in (0,1)$),
		then for an arbitrary $\vec m \in \R^N_+$ there exists
		a minimal generalized $N$-cluster $\vec E$ with $\vec m^\infty(\vec E)=\vec m$.
	\end{theorem}
	\begin{proof}
		Let us define over $X^\infty$ the distance 
		\[
		d^\infty(x,y) 
		:= \begin{cases}
		\arctan d(x',y'), &\text{if $x=(i,x')$ and $y=(i,y')$ for some $i\in \N$},\\
		\frac{\pi}{2}, & \text{otherwise},
		\end{cases}
		\]
		the measure $\mu^\infty$ on $X^\infty$ by
		\[
		\mu^\infty(E) := \sum_j \mu(\phi_j^{-1}(E)), \qquad E \subset X^\infty
		\]
		and define $G^\infty = \Z\times G$ the product group acting over $X^\infty$ by
		\[
		(n,g)(m,x) = (n+m, g(x))\qquad g\in G, x\in X, n,m\in \Z.
		\]
		
		Notice that $P^\infty$ and $\mu^\infty$ satisfy monotonicity, superadditivity, $G$-invariance 
		and compactness properties of section~\ref{sec:notation}.
		The semicontinuity property is simply an application of Fatou Lemma and 
		Beppo Levi property is an application of Beppo Levi theorem. 
		
		We claim that (i) of Theorem~\ref{th:comp_comp} holds.
		In fact if $K^\infty$ is a compact set in $X^\infty$ then the set $J:=\{j\in \Z\colon (j,x)\in K^\infty\ \text{for some $x\in X$}\}$
		is finite.
		Let $K:= \{x\in X\colon (j,x)\in K^\infty\text{ for some $j\in J$}\}$. 
		Since $J$ is finite $K$ is compact in $X$.
		Then the set 
		\[
		\{(n,g)\in G^\infty\colon (n,g)K^\infty \cap K^\infty \neq \emptyset\}
		\subset (J-J)\times \{g\in G\colon gK\cap K\neq \emptyset\}
		\] 
		is finite because if $(j,x)\in(n,g)K^\infty\cap K^\infty$ then $x\in K$, $g(x)\in K$, $j\in J$ and $n+j\in J$.
		
		Also the condition (ii) of Theorem~\ref{th:comp_comp} is satisfied for the metric measure space $X^\infty$ with 
		this group action. 
		In fact the set $U^\infty:=\{0\}\times U$ satisfies 
		$G^\infty \overline {U^\infty} = X^\infty$ 
		and $\mu^\infty(g\overline {U^\infty} \cap \overline{U^\infty})=0$ for all $g\in G^\infty$ except when $g$
		is the neutral element in $G^\infty$.
		
		For condition (iii) of Theorem~\ref{th:comp_comp} just notice that if $\mu^\infty(E)<\eps$ implies that 
		$\mu(E^j)< \eps$ (where $E^j=\phi_j^{-1}(E)$)
		hence $P(E^j,V)\ge f(\mu(E_j\cap V))$ and summing up, using the subadditivity 
		of $f$,
		\begin{align*}
		P^\infty(E,\{0\}\times V) &
		\ge \sum_j f(\mu(E_j\cap V)) \\
		&\ge f\left(\sum_j \mu(E_j\cap V)\right)
		=  f(\mu^\infty(E\cap \{0\}\times V)).
		\end{align*}
		
		Let $\vec E_k$ be a sequence of generalized $N$-clusters in $X$ satisfying $\vec m^\infty(\vec E_k) = \vec m\in \R^N$.
		By Theorem~\ref{th_comp_clust} applied to $\vec E_k$, $X^\infty$, $\mu^\infty$, $P^\infty$
		in place of $X$, $\mu$, $P$ respectively,
		we have the existence $N$-clusters $\vec E^i\in X^\infty$ for $i\in I$ with $I$ at most countable,
		such that 
		\begin{equation}
		\label{eq:min_gen_clust1}	
		\sum_i \vec m^\infty(\vec E^i) = \vec m,
		\qquad 
		\sum_i P^\infty(\vec E^i) \le \liminf_k P^\infty(\vec E_k).	
		\end{equation}
		
		Given an injective map $f\colon I\times \Z\to \Z$ we can define the 
		generalized $N$-cluster $\vec E=(E_1,\dots, E_N)$ in $X$ such that 
		\[
		\phi_{f(i,j)}^{-1}(E_n) = \phi_j^{-1}(E^i_n), \qquad 
		\text{for every $n=1,\dots N$, $j\in \Z$, $i\in I$.} 
		\]
		Hence \eqref{eq:min_gen_clust1} reads now 
		\[
		\vec m^\infty(\vec E) = \vec m,
		\qquad 
		P^\infty(\vec E) \le \liminf_k P^\infty(\vec E_k)
		\]
		which means that $\vec E$ is a generalized minimal $N$-cluster in $X$.
	\end{proof}
	
	\begin{remark}
		Note that a minimal \textit{generalized}  cluster can be thought in many particular cases as a kind of natural relaxation of
		the notion of a minimal cluster. In fact, if it is possible, say, to cut away from each page everything outside of a sufficiently large ball, without changing too much the perimeter and the volumes, so that what remains on each a page is a bouunded cluster, then put by the group action all the bounded clusters from each page to just one page, and finally adjust the volumes, say, by adding small isoperimetric sets or at least sets with sufficiently small perimeters, we get that for every $\varepsilon>0$ and for each generalized c $N$-cluster $\vec E$ there is an $N$-cluster $\vec E'$ with the same volumes, i.e. $\vec m^\infty (\vec E')= \vec m (\vec E)$ and \[P(\vec E')\leq P^\infty(\vec E) +\varepsilon.\] This can be done for instance in a Euclidean space (the cutting of a large bounded set from each page without changing too much the volumes and the perimeter can be done in view of the coarea formula; in more general spaces even a coarea type inequality would suffice).       
	\end{remark}

	\section{Basic examples}\label{sec_ClusterExamples}
	
	For merely illustrative purposes we provide here several examples of existence of finite or infinite generalized isoperimetric clusters 
	in a finite-dimensional vector space (for any anisotropic perimeter related to some norm), in the hyperbolic plane and in the Heisenberg groups.
	Note that even such simple examples in fact provide immediately the existence resluts for the whole range of equivalent perimeters
	in the mentioned spaces, without requiring the study of regularity properties of such clsuters and/or just isoperimetric sets for each particular perimeter. It is also worth mentioning that in the same way one can formulate similar existence results in many more interestng geometries (in particular, in higher dimensional hyperbolic spaces, or more general Carnot groups). 
	
	\subsection{Finite-dimensional space}
	Let $X$ be an $n$-dimensional finite dimensional space equipped with any norm $\|\cdot\|$, $\mu$ be the Lebesgue measure on $X$,  
	$G:=\Z^n$ acting by translations, and $P$ be the relative perimeter functional
	in $X$ corresponding to the chosen norm, i.e.\
	\[
	P(E,U) := \sup
	\left\{
	\int_U \mathbf{1}_E \mathrm{div} v\, d\mu \colon 
	\|v\|_\infty\leq 1
	\right\},
	\]
	the $\sup$ being taken over smooth vector fields $v$ with compact support on $H$, with
	$\|v\|_\infty:= \sup_{p\in X} \|v(p)\|$.
	The following corollary is a direct consequence of Theorem~\ref{th_exist_generclust1}.
	
	\begin{corollary}\label{co_existclust_findim1}
		For every $\vec m \in \R^N_+$ there exists
		a minimal generalized $N$-cluster $\vec E$ in $X$ with $\vec m^\infty(\vec E)=\vec m$.
	\end{corollary}
	
	It is worth mentioning that there is nothing obligatory in the choice of the group $G=\Z^n$; instead, another some other crystallographic group could have been chosen.

	\subsection{Hyperbolic space}
	
	Taking $H$ to be the hyperbolic plane, $\mu$ be its canonical volume measure,  
	$G$ be any countable Fuchsian group (i.e.\ a discrete subgroup of isometries of $H$) acting properly discontinuously and cocompactly on $H$
	(e.g. one may take $G$ to be the classical Fucsian group providing the tiling of $H$ into isometric Schwartz triangles)
	and $P$ be the classical Riemannian relative perimeter functional
	in $H$, i.e.\
	\[
	P(E,U) := \sup
	\left\{
	\int_U \mathbf{1}_E \mathrm{div}_H v\, d\mu \colon 
	\|v\|_\infty\leq 1
	\right\},
	\]
	the $\sup$ being taken over smooth vector fields $v$ with compact support on $H$, with
	$\|v\|_\infty:= \sup_{p\in H} |v_p|_p$, $|\cdot|_p$ standing for the Riemannian norm of a vector in $T_p H$.
	
	\begin{corollary}\label{co_existclust_hyperb1}
		For every $\vec m \in \R^N_+$ there exists
		a minimal generalized $N$-cluster $\vec E$ in $H$ with $\vec m^\infty(\vec E)=\vec m$.
	\end{corollary}
	
	\begin{proof}
		Note that the perimeter functional $P$ clearly satisfies the semicontinuity, monotonicity, superadditivity, compactness, $G$-invariance and Beppo Levi properties,
		and 
		conditions~(i) and~(ii) of Theorem~\ref{th:comp_comp} are satisfied as well 
		with $B$ given by Lemma~\ref{lm_Voronoi1}. 
		Condition~(iii) of Theorem~\ref{th:comp_comp} is satisfied as in Remark~\ref{rm_isoperim1} with $V\supset B$ being a sufficiently large ball containing $B$, 
		and $\alpha := 1/2$, $\varepsilon < \mu(V)/2$ due to the relative isoperimetric inequality over a compact Riemannian manifold 
		$\bar V$~\cite{keselman05-isopermhyperb1}. The claim is now a direct application of Theorem~\ref{th_exist_generclust1}.
	\end{proof}
	
	\begin{remark}\label{rm_existclust_hyperb1}
		If, chosen a Fuchsian group $G$, there is an equivalent Finslerian structure on $H$ (with the norm $\|\cdot\|_p$ in each $T_p H$ equivalent to the Riemannian one $|\cdot|_p$),  invariant under the action of $G$, then the same argument gives the  existence for every $\vec m \in \R^N_+$
		of a minimal generalized $N$-cluster $\vec E$ in $H$ with $\vec m^\infty(\vec E)=\vec m$, for the perimeter $P_G$ relative to this Finslerian structure instead of $P$, i.e.\ for 
		\[
		P_G(E,U) := \sup
		\left\{
		\int_U \mathbf{1}_E \mathrm{div}_H v\, d\mu \colon 
		\|v\|_{G,\infty}\leq 1
		\right\},
		\]
		the $\sup$ being taken over smooth vector fields $v$ with compact support on $H$, with
		$\|v\|_{G,\infty}:= \sup_{p\in H} \|v_p\|_p$. To see this it is enough to note that $P_G$ and $P$ are equivalent.	
	\end{remark}
	
	\subsection{Heisenberg groups}
	
	Taking $H^n$ to be the Heisenberg group of topological dimension $2n+1$, 
	$\mu$ be its Haar measure,  
	$G$ be the respective discrete Heisenberg group (i.e. once $H^{n}$ is canonically associated with a group of matrices with real entries,
	$G$ is associated with the subgroup of such matrices with integer entries), $\{X_i\}_{i=1}^{2n}$ be left-invariant vector fields
	satisfying the H\"{o}rmander condition,
	and $P$ be the sub-Riemannian relative perimeter functional 
	corresponding to the choice of $X_i$,
	defined as the total variation measure of the vector ($\R^{2n}$-valued) measure
	\[
	D_X \mathbf{1}_1:= (D_{X_1} \mathbf{1}_E, \ldots, D_{X_{2n}} \mathbf{1}_E),
	\]
	where the distribution $D_{X_i} f$ is defined 
	by its action $\langle \varphi, D_{X_i} f\rangle$ 
	on every test function $\varphi_\in C_0^\infty(\R^{2n})$
	by the formula
	\[
	\langle \varphi, D_{X_i} f\rangle := - \int_{\R^{2n}} f X_i\cdot \nabla \varphi\, dx  - \int_{\R^{2n}} f \varphi \mathrm{div}\,X_i\, dx. 
	\]
	Note that there are several equivalent definitions of
	this perimeter given by theorem~3.1 of~\cite{AmbrosioGhezziMagn15-BVsubriem}. We also observe that one of the equivalent definitions
	is given in~[section 5.3]\cite{Miranda03-BVgoodMMS} and used in particular in~\cite{DanGarofNhieu98-CarnotTraceIneq}.
	
	\begin{corollary}\label{co_existclust_heisenb1}
		For every $\vec m \in \R^N_+$ there exists
		a minimal generalized $N$-cluster $\vec E$ in $H^n$ with $\vec m^\infty(\vec E)=\vec m$.
	\end{corollary}
	
	\begin{proof}
		We equip $H^n$ 
		with either the Carnot-Caratheodory distance
		or any equivalent left-invariant distance $d$, so that now $G$ acts by isometries.
		The claim follows from Theorem~\ref{th_exist_generclust1} since he perimeter functional $P$ satisfies the semicontinuity, monotonicity, superadditivity, Beppo Levi and $G$-invariance properties, as well as 
		conditions~(i) and~(ii) of Theorem~\ref{th:comp_comp} (with $B$ given by Lemma~\ref{lm_Voronoi1}),
		while 
		\begin{itemize}
			\item condition~(iii) of Theorem~\ref{th:comp_comp} is satisfied as in Remark~\ref{rm_isoperim1} with $V\supset B$  sufficiently large ball containing $U$, 
			and $\alpha := (Q-1)/Q$, $Q:=2n+2$ standing for the homogeneous dimension of $H^n$, $\varepsilon < \mu(V)/2$ due to the relative isoperimetric inequality in a ball of a Carnot-Caratheodory space (theorem~1.6 from~\cite{DanGarofNhieu98-CarnotTraceIneq}),
			\item compactness property is given by theorem~3.7 from~\cite{Miranda03-BVgoodMMS}. 
		\end{itemize}
	\end{proof}
	
	\begin{remark}\label{rm_existclust_heisenb1}
		Similarly to Remark~\ref{rm_existclust_hyperb1} if one takes in $H^n$ any perimeter
		$P$ defined as $P(E,U):=|D\mathbf{1}_E| (U)$ with the metric total variation $u\mapsto |Du|$ defined with respect to any chosen left $G$-invariant distance in $H^n$, then word-to-word repetition of the proof of the above Corollary~\ref{co_existclust_heisenb1} gives
		the  existence for every $\vec m \in \R^N_+$
		of a minimal generalized $N$-cluster $\vec E$ in $H$ with $\vec m^\infty(\vec E)=\vec m$ for such a perimeter.	
	\end{remark}

	\appendix
	
	\section{Useful facts about group actions}
	
	We collect here the following lemma, useful in applications of our results to verify condition~(ii) of Theorem~\ref{th:comp_comp}, and some easy remarks related to it.

	\begin{lemma}\label{lm_Voronoi1}
	Let $(X,d,\mu)$ be a complete locally compact metric space with nonnegative $\sigma$-finite Radon  measure.
	If $G$ is a 
	topological group acting 
	on $X$ by homeomorphisms 
	preserving $\mu$-nullsets, and 
	\begin{itemize}
		\item[(i)] the action of $G$ is proper discontinuous,
		\item[(ii)] the set of $x\in X$ having nontrivial stabilizer subgroup
		$G_x:=\{g\in G\colon gx=x\}$, i.e.\ with $G_x\neq \{1\}$, is $\mu$-negligible, i.e.\
		$\mu(S)=0$,
		\item[(iii)] and $X$ is compactly generated, i.e. there exists a compact $K\subset X$ such that $GK=X$.
	\end{itemize}
	Then there
	is a precompact Borel set $B\subset X$ (a fundamental domain for the action of $G$)
	such that $G\bar B=X$, $\mu(\partial B)=0$,
	and
	$\mu(gB \cap B)=0$. Moreover, if the action of $G$ on $X$ is free, then one can choose $B$ so as to have $G B=X$ and
	$gB \cap B=\emptyset$ for all $g\in G$, $g\neq 1$.
\end{lemma}


\begin{remark}\label{rem_cocomp1}
	It is a more or less folkloric fact that the condition~(iii) of Lemma~\ref{lm_Voronoi1} for locally compact metric space $X$ 
	is equivalent to \textit{cocompactness} of the action of $G$, that is, to compactness of the quotient space $X/G$.
	In fact, since the natural projection map $\pi \colon X\to X/G$ is continuous, then condition~(iii) implies that $X/G=\pi(K)$ is compact. Vice versa, note that $\pi$ is also an open map (as a projection map under group action, namely, because for every open
	$U\subset X$ one has $\pi^{-1}(\pi(U)))=\cup_{g\in G} gU$ is open since so is each $gU$). Therefore, taking a cover $\{U_\lambda\}$ of $X$ by precompact open sets, we have that $\{\pi(U_\lambda)\}$ is an open cover of $X/G$, and if $X/G$ is compact, extracting a
	finite subcover    $\{\pi(U_{\lambda_j})\}_{j=1}^m$, $m\in \N$ of $\{\pi(U_\lambda)\}$, we get that
	\[
	K:=\bigcup_{j=1}^m \bar U_{\lambda_j}
	\] 
	is a compact set satisfying $GK=X$.
\end{remark}

\begin{remark}\label{rem_cocomp2}
	Another immediate observation worth mentioning is that for the condition~(iii) of Lemma~\ref{lm_Voronoi1} 
	to hold, it is necessary that $G$ be infinite, unless, of course, $X$ is compact. 
\end{remark}

\begin{proof}	
	For the readers' convenience we provide a detailed proof divinding it in two steps.
	
	{\sc Step 1}. We show the existence for each $x\in X\setminus S$ of an open precompact ball $U_x\subset X$ with
	$x\in U_x$ such that
	\[
	gU_x\cap U_x=\emptyset\quad\mbox{for all $g\in G$, $g\neq 1$}
	\]
	and $\mu(\partial U_x)=0$.
	Let now
	\[
	D_x:=\{ y\in X\colon d(y,x)< d(gy, x)\quad\mbox{for all $g\in G$, $g\neq 1$}\}
	\}
	\]
	stand for the Dirichlet-Voronoi fundamental domain for the action of $G$, so that
	\[
	g D_x\cap D_x=\emptyset\quad\mbox{for all } g\in G, g\neq 1.
	\]
	In fact, if $y\in g D_x$, that is, $gy\in D_x$, then $y:=g^{-1}gy \not\in D_x$.
	
	Since $X$ is locally compact, then there is an $r>0$ such that the balls $B_\varepsilon(x)$ with $\varepsilon<r$ are precompact.
	We claim that $D_x$ contains a small ball $B_\rho(x)$, $\rho<r$. In fact, otherwise there is a sequence of points $y_k\in X$ and a sequence of
	$g_k\in G$  such that
	$\lim _k y_k = x$ and 
	\[
	d(y_k,x)\geq d(g_k y_k, x) 
	\]
	so that 
	also $\lim_k g_k y_k = x$. 
	This is only possible when the set 
	$\{g_k\}$ is finite, because when $\varepsilon <r$ one has that
	$g \bar B_\varepsilon (x)\cap \bar B_\varepsilon (x)\neq \emptyset$ for only a finite number of $g\in G$, again by proper discontinuous action. 
	Therefore, there is a $g\in G$, $g\neq 1$, such that up to a subsequence (not relabeled) we have
	$\lim _k g y_k = x$, which imples that $gx=x$, that is a contradiction to $x\not\in S$. 
	
	It remains now to observe that since $B_\rho(x)\subset  D_x$, we may find a small ball $U_x:= B_s(x)$  such that
	$U_x\subset B_\rho(x)$ (hence is precompact) and $\mu(\partial U_x)=0$ (all but a countable number of $s<\rho$ will suffice).

	{\sc Step 2}. If $S\neq\emptyset$, we let for every $k\in \N$ let $S_k\supset S$ be an open set such that $\mu(\bar S_k)\leq 1/k$ 
	(otherwise, if $S=\emptyset$, we just let $S_k:=\emptyset$).
	Consider a finite cover $\{S_k, U_1,\ldots, U_{n_k}\}$  of $K$ in $X$, where 
	each  $U_i=U_{x_i}$ is a precompact open 
	sets containing an $x_i\in K\setminus S$ such that
	\[
	gU_i\cap U_i=\emptyset\quad\mbox{for all $g\in G$, $g\not\in G_x$}
	\]
	and $\mu(\partial U_i)=0$ for all $i$.
	Without loss of generality we may assume 
	that
	\begin{equation}\label{eq_Ujpart1}
	U_j\subset \left(\bigcup_{i=1}^{n_k} U_i\right)\cup S_k
	\end{equation} 
	for all $j>n_k$.
	We define inductively disjoint sets
	\begin{align*}
	V_1 &:= U_1,\\
	V_{i+1} &:= U_{i+1}\setminus \bigcup_{j=1}^i G U_j.
	\end{align*}
	Note that since $\bar U_j$ are all compact and $G$ acts properly discontinuously, then  
	each $G U_j$ in the definition of $V_i$ can be substituted by a finite union of $g U_j$ over a finite subset
	of $g\in G$ (depending of course on $i$ and $j$): in fact,
	\[
	g U_j\cap U_{i+1} \subset g \overline{(U_j\cup U_{i+1})}\cap \overline{(U_j\cup U_{i+1})}\neq \emptyset  
	\]
	for an at most finite set of $g\in G$. Therefore, recalling also that
	$\partial (g U_j)= g\partial U_j$ (since the action of $g$ is a 
	homeomorphism), and
	hence $\mu (\partial (g U_j))= \mu (g\partial U_j)=0 $ (since the action preserves $\mu$-nullsets), we get
	$\mu(\partial V_i)=0$ for all $i$.
	Seting now
	\[
	\tilde B_k:=\bigcup_{i=1}^{n_k} V_i,
	\]
	we get that $\tilde B_k$ is a Borel set satisfying 
	\begin{equation}\label{eq_deftildeB1}
	g\tilde B_k \cap \tilde B_k=\emptyset\quad\mbox{for all $g\in G$, $g\neq 1$}	
	\end{equation}
	because so are all $V_i$, $\mu(\partial \tilde B_k)=0$.
	We now define
	\[
	B_k:= 	\tilde B_k \cup S_k \quad\text{and } B:=\cap_k B_k.
	\]
	Note that the sequence of sets $B_k$ is nonincreasing in view of~\eqref{eq_Ujpart1}.
	By~\eqref{eq_deftildeB1} one has 	$gB_k \cap B_k\subset S_k\cup g S_k$.
	Thus
	\[
	gB\cap B \subset \bigcap_k (S_k\cup gS_k)
	\]
	for all $g\in G$, $g\neq 1$.
	Since $S_k\cup gS_k$ is a decreasing sequence of sets and
	\[
	\mu(S_k\cup gS_k)\leq 	\mu(S_k) + \mu(gS_k)\to \mu(S) + \mu(gS)=0 
	\]
	as $k\to\infty$, since the action of $G$ preserves $\mu$-nullsets, then
	$\mu(gB\cap B) =0$
	for all $g\in G$, $g\neq 1$.
	
	Further, 
\begin{align*}
	\mu (\partial B)\leq  \mu((\partial B)\setminus \bar S_k) +  \mu(\bar S_k) \leq \mu (\partial \tilde B_k) + \mu(\bar S_k) = \mu(\bar S_k) \leq 1/k,
\end{align*}
	hence $\mu(\partial B)=0$.
	Finally,
	$GB_k=X$ because $GB_k \supset GK=X$.
	This means that for each $x\in X$ there is an $y_k \in B_k$ and $g_k\in G$ such that
	$g_ky_k=x$. Since  $B_k$ is a decreasing sequence of precompact sets, then there is a compact $\tilde K$ such that $\{y_k\}\subset \tilde K$. Thus $g_k\tilde K$ have a common point which can only happen for a finite number of $g_k$ by discontiuity of action. Thus, up to a subsequence we have $g_k=g\in G$, and $gy_k=x$. Extracting a further convergent subsequence of
	$y_k$, we have $\lim_k y_k=y\in \bar B$ and $gy=x$, hence $G\bar B=X$.
	Note that
	when $S=\emptyset$, then all $S_k=\emptyset$ and hence $B=B_k=\tilde B_k$ for some finite $k$, concluding the proof.
\end{proof}

	The following easy technical  lemma is also used.
	
	\begin{lemma}\label{lm_groupPart1}
		Let $G$ be a group acting on a set $X$ and $K\subset X$.
		If
		\[
		\nu:=\#\{g\in G\colon gK\cap K\neq \emptyset\} <\infty,
		\]
		then there is a partition of $G$ into at most $\nu$ disjoint subsets $\mathcal{F}_j$,  such that 
		$hK\cap gK =\emptyset$
		for every couple
		$\{h, g\}\subset \mathcal{F}_j$ for each
		$j=1,\ldots,\nu$, and each 	$\mathcal{F}_j$ is a maximal subset of $G$ with this property, i.e.\ for every 
		$\tilde h\not \in \mathcal{F}_j$ there is a $\tilde g\in \mathcal{F}_j$ such that
		$\tilde h K\cap \tilde g K \neq\emptyset$.
	\end{lemma}
	
	\begin{proof}
		We construct the sets $\mathcal{F}_j$ inductively.
		Namely, by Zorn's lemma there is a set $\mathcal{F}_1\subset G$  maximal (with respect to inclusion)  containing $1$ such that
		$hK\cap gK =\emptyset$
		for every couple
		$\{h, g\}\in \mathcal{F}_1$.
		Once $\mathcal{F}_j$ are constructed for $j=1,\ldots, k$, if there  an $h\not\in \cup_{j=1}^k \mathcal{F}_j$, we take
		$\mathcal{F}_{k+1}$ to be a maximal (with respect to inclusion) subset of $G$ containing $h$ such that
		$hK\cap gK =\emptyset$
		for every couple
		$\{h, g\}\in \mathcal{F}_{k+1}$ (the existence is again guaranteed by Zorn's lemma).
		It is easy to see that the induction finishes after at most $\nu$ steps, i.e.\ we may construct at most $\nu$ nonempty sets 
		$\mathcal{F}_j$. 
		In fact, if we managed to construct nonempty sets $\mathcal{F}_1,\ldots, \mathcal{F}_{\nu+1}$,
		then by maximality of each of them, for an $h\in \mathcal{F}_{\nu+1}$ there are $h_k\in \mathcal{F}_k$, such that 
		$h_k K \cap h K \neq \emptyset$, or, equivalently,  $h^{-1}h_k K \cap K\neq \emptyset$,
		where $k=1,\ldots,\nu$. In other words, 
		$\{h^{-1}h_k\}_{k=1}^{\nu} \subset \{g\in G\colon gK\cap K\neq \emptyset\}$,
		and since all $h^{-1}h_k$ are distinct, these two sets must coincide. Hence, for some $k$ one has  $h^{-1}h_k=1$, that is,
		$h=h_k$, which is impossible, and thus only sets $\mathcal{F}_1,\ldots, \mathcal{F}_\nu$ are nonempty.
	\end{proof}

	\section{Auxiliary lemmata}
	
	We collect here, mainly for the sake of completeness and readers' convenience, some auxiliary lemmata of more or less folkloric nature.
	
	\begin{lemma}[equisummability]\label{lm:equisummability}%
		Suppose that $m_{k,i}\geq 0$, and: 
		\begin{gather*}
		\lim_{k\to +\infty} \sum_i m_{k,i} = m, \\
		\lim_k m_{k,i} = m_i\\
		\lim_n \enclose{\sup_k \sum_{i=n}^{+\infty} m_{k,i}} = 0.
		\end{gather*}
		Then 
		\[
		\sum_i m_i = m.
		\]
	\end{lemma}
	\begin{proof} For every $\eps>0$ there is an $n_\eps\in \N$ such that 
		for all $k\in \N$ and $n\geq n_\eps$ one has 
		\[
		\sum_{i=n}^{+\infty} m_{k,i} < \eps , \quad\text{and}\quad  \sum_{i<n} m_{k,i}\le \sum_i m_{k,i} \le  \sum_{i<n} m_{k,i}+\eps.
		\]
		Hence either $m=+\infty$ and  $m_i=+\infty$ for some $i<n_\eps$, or $m\in\R$.
		In the first case the thesis follows. In the second letting $k\to +\infty$ we obtain
		for $n\geq n_\eps$ the estimate
		\[		m-\eps \le \sum_{i=1}^n m_i \le m.
		\]
		Letting now $n\to +\infty$, one gets 
		\[
		m-\eps \le \sum_{i} m_i \le m
		\] 
		for all $\eps>0$ implying the thesis.
	\end{proof}
	
	\begin{lemma}\label{lm:union}
		Suppose that for all $k\in \N$, $j=1,\ldots, N$, where $N\in \N\cup\{\infty\}$ the sets 
		$A_k\subset X$, $A_k^j\subset X$ be $\mu$-measurable  and
		\[
		\mu\left(A_k \triangle \bigcup_{j=1}^N A_k^j\right)=0, \qquad \mu\left(A_k^j \cap A_k^i\right)=0 \quad \mbox{whenever } i\neq j.
		\]
		If $A_k\to A$ and $A_k^j\to A^j$ in $L^1_\loc(\mu)$ as $k\to\infty$, then 
		\[
		\mu\left(A\triangle \bigcup_{j=1}^N A^j\right)=0, \qquad \mu\left(A^j \cap A^i\right)=0 \quad \mbox{whenever } i\neq j.
		\]
	\end{lemma}
	
	\begin{proof}
		It suffices to write for every compact $K\subset X$ the relationships
		\begin{align*}
		\mu\left(\left( A_k \triangle \bigcup_{j=1}^N A_k^j\right) \cap K\right) &=\int_K \left|\mathbf{1}_{A_k}(x)-\left( \sum_{j=1}^N \mathbf{1}_{A_k^j}(x)\right)\right| \,d\mu(x),\\
		\mu\left(\left( A_k^j \cap A_k^i\right) \cap K\right) &=\int_K \mathbf{1}_{A_k^j}(x)\mathbf{1}_{A_k^i}(x) \,d\mu(x),
		\end{align*}
		and pass to the limit as $k\to \infty$.
	\end{proof}

	\bibliographystyle{plain}

\end{document}